\documentclass[12pt]{article}

\usepackage{amsmath}
\usepackage{amssymb}
\usepackage{amsthm}
\usepackage{tikz}
\usepackage{tikz,pgfplots}
\usetikzlibrary{decorations.markings}
\usepackage{comment}

\usepackage{graphicx,tikz}
\usetikzlibrary{shadows}
\usepackage{subcaption}
\usepackage{todonotes}
\usepackage{graphicx}
\usetikzlibrary{arrows}
\usepackage[utf8]{inputenc}
\usepackage[english]{babel}
\usepackage{array}
\usepackage{booktabs}
\usepackage{mdframed}
\usepackage{physics}
\usepackage{url}
\usepackage{xcolor}
\usepackage{enumitem}
\usepackage{soul}
\setstcolor{red}

\usepackage[noend,linesnumbered,ruled,vlined]{algorithm2e}
\SetKw{Break}{break}
\SetKw{Not}{not}
\SetKw{AND}{and}
\DontPrintSemicolon
\usepackage{xcolor}

	\newtheorem{theorem}{Theorem}
	\newtheorem{definition}[theorem]{Definition}
	\newtheorem{corollary}[theorem]{Corollary}
	\newtheorem{lemma}[theorem]{Lemma}

	\newtheorem{proposition}[theorem]{Proposition}

	\newcommand{\equ}{{\rm eq}}

	\newcommand{\diam}{{\rm diam}}

\begin{document}

\title{The clique number of the exact distance $t$-power graph: complexity and eigenvalue bounds}

\author{
Aida Abiad
\thanks{\texttt{a.abiad.monge@tue.nl}, Department of Mathematics and Computer Science, Eindhoven University of Technology, The Netherlands}
\and
Afrouz {Jabal Ameli}\thanks{\texttt{a.jabal.ameli@tue.nl}, Department of Mathematics and Computer Science, Eindhoven University of Technology, The Netherlands} 
\and
Luuk Reijnders
\thanks{\texttt{l.e.r.m.reijnders@student.tue.nl}, Department of Mathematics and Computer Science, Eindhoven University of Technology, The Netherlands}
}

\date{}
\maketitle

\begin{abstract}
The exact distance $t$-power of a graph $G$, $G^{[\sharp t]}$, is a graph which has the same vertex set as $G$, with two vertices adjacent in $G^{[\sharp t]}$ if and only if they are at distance exactly $t$ in the original graph $G$. We study the clique number of this graph, also known as the $t$-equidistant number. We show that it is NP-hard to determine the $t$-equidistant number of a graph, and that in fact, it is NP-hard to approximate it within a constant factor. We also investigate how the $t$-equidistant number relates to another distance-based graph parameter; the $t$-independence number. In particular, we show how large the gap between both parameters can be. The hardness results motivate deriving eigenvalue bounds, which compare well against a known general bound. In addition, the tightness of the proposed eigenvalue bounds is studied.
\end{abstract}	


\section{Introduction}
A set \textcolor{black}{$S$} of vertices of a graph $G$ is a \emph{$t$-equidistant set} if $d(u,v) = t$ for all $u,v \in \textcolor{black}{S}$, where $u \not = v$. A set $\textcolor{black}{S} \subseteq V(G)$ is an \emph{equidistant set} if there exists $t \in \mathbb{N}$ such that $\textcolor{black}{S}$ is a $t$-equidistant set; we will say that a $t$-equidistant set $\textcolor{black}{S}$ has \emph{diameter} $t$, which we will denote by $\diam(\textcolor{black}{S}) = t$. The \emph{$t$-equidistant number} of a graph $G$ is the number of vertices in a largest $t$-equidistant set of $G$ and the \emph{equidistant number} $\equ(G)$ is defined by $\equ (G) = \max \{ \equ_t(G): 1 \leq t \leq \diam(G)\} $. A \emph{strict equidistant set} $\textcolor{black}{S}$ is an equidistant set such that if $\textcolor{black}{S}$ induces a clique then $|\textcolor{black}{S}| = 1$; the \emph{strict equidistant number} $\equ'(G)$ is the size of a largest strict equidistant set in $G$.

For a positive integer $t$, the \emph{power graph} of $G$, denoted by $G^t$, is a graph which has the same vertex set as $G$, with two vertices adjacent in $G^t$ if and only if they are at distance at most $t$ in the original graph $G$. A similar, but less studied, notion is the \emph{exact distance $t$-power} of a graph $G$. This is denoted by $G^{[\sharp t]}$, and is a graph which has the same vertex set as $G$, with two vertices adjacent in $G^{[\sharp t]}$ if and only if they are at distance exactly $t$ in the original graph $G$. One of the first explicit definitions of this concept is due to Simi\'c \cite{S1983}, who studied graphs which had an exact distance power isomorphic to their line graph. Other work includes the investigation of graphs that are isomorphic to their exact square \cite{AF2017} and the structure of exact distance powers of products of graphs \cite{BGKT2019}. Also, the notion of exact distance powers often appears in the context of the theory of sparse graphs, see \cite[Section 11.9]{bookNO2012}. The main focus in earlier investigations of exact distance graphs was on their chromatic number, see for instance \cite{HKQ2018, Q2020, FHMNNSV2021, la_2-distance_2022}.

Using the exact distance $t$-power graph we have an alternative definition of the $t$-equidistant number: $\equ_t(G) = \omega(G^{[\sharp t]})$. In this work we focus on investigating the clique number of the exact distance $t$-power of a graph, a parameter that has been investigated before. Instances of it are the work by Foucaud et al.\ \cite{foucaud_cliques_2021}, who studied the problem of finding the maximum possible clique number among exact distance $t$-powers of graphs of given maximum degree. Haemers \cite{H1997} studied a similar parameter related to the $t$-equidistant number. 

In this paper we show that the problem of determining the ($t$-)equidistant number is NP-hard, and also that it is impossible to approximate the $t$-equidistant number with a constant factor unless $P = NP$. We also investigate the relation between the equidistant number and the well studied \emph{$t$-independence number} of a graph $G$, which is the maximum size of a set of vertices in $G$ at pairwise distance greater than $t$. In particular, we focus on understanding how large the gap between both parameters can be. The hardness results motivate deriving eigenvalue bounds, since these can be computed in polynomial time. Using the results on the parameters' relationship and eigenvalue interlacing, we derive several eigenvalue bounds on the $t$-equidistant number. We show that, for several instances, the obtained bounds compare favorably to previously known bounds. Finally, we study the tightness of the proposed eigenvalue bounds.

\subsubsection*{Some applications of the parameter $\equ_t(G)$}

The parameter of interest, $\equ_t(G)$, has a direct application in the field of Information Theory. Let $A_q = \{0,\hdots,q-1\}$ for an integer $q \geq 2$. A non-empty subset $C \subset A_q^n$ is called a $q$-\emph{ary code}. If $\abs{C} =N$ and $d$ is the minimum Hamming distance in $C$, then $C$ is called a $(n,d,N)_q$-\emph{Hamming code}. By $A_q(n,d)$ we denote the largest $N \in \mathbb{N}$ for which there exists a $(n,d,N)_q$-Hamming code. If the Hamming distance between any codewords in $C$ is exactly $d$, then the code $C$ is called \emph{equidistant}.
 A quantity of interest in the study of $q$-ary codes is $A_q(n,d)$, which is the largest $N \in \mathbb{N}$ for which there exists a $(n,d,N)_q$-Hamming code. A similar quantity exists for equidistant codes specifically, $B_q(n,d)$ is the largest $N \in \mathbb{N}$ for which there exists an equidistant $(n,d,N)_q$-Hamming code. Much like the $(t-1)$-independence number of the hypercube graph $Q_n$ is equal $A_2(n,t)$ \cite{abiad_eigenvalue_2023}, the $t$-equidistant number of $Q_n$ is equal to $B_2(n,t)$. Hence, the spectral bounds on the equidistant number presented here can be applied in the context of equidistant codes.
 Equidistant codes have been extensively studied, see for instance \cite{B2007, L1973}, and their 
 interest is a consequence of their application in random network coding \cite{KK2008}. 

Another application is to \emph{equidistant subspace codes}, i.e., subspace codes with the property that the intersection of any pair of codewords has the same dimension. Equidistant subspace codes were shown to have relevant applications in distributed storage in \cite{RE}.

\section{Complexity of $\equ_t(G)$ and $\equ(G)$}\label{sec:complexity}
 
We investigate the complexity of the ($t$-)equidistant number, and show that it is NP-hard to determine both $\equ_t$ and $\equ$.

First, we will show that it is NP-hard to determine $\equ_t$. We do this by constructing auxiliary graphs where edges in the original graphs become paths of length $t$. We remark that the gadget used for this reduction for odd values of $t$ resembles the gadget used for the NP-hardness proof of $\omega(G^t)$ by Lin and Skiena \cite[Theorem 5.2]{LS1995}. Finally, we show that it is NP-hard to approximate $\equ_t$ within a constant factor.

We begin this section by exploiting some properties of a well-studied family of graphs known as split graphs. Such properties will be useful in particular for our reduction when $t$ is even.

    \begin{definition}[Split graph]
    A graph $G = (V,E)$ is called \emph{split} if there exists a partition of the vertices into subsets $V_1, V_2$ such that $V_1$ is a clique and $V_2$ is an independent set.
    \end{definition}

Split graphs were defined and characterized by F\"{o}ldes and Hammer \cite{FH77v1, FH77v2}. These graphs were also introduced independently by Tyshkevich and Chernyak \cite{CC86} under the name \emph{polar graphs}.

    \begin{lemma}\label{lem:SplitCheck}
        Whether or not a graph $G$ is split can be determined in polynomial time. Furthermore if $G$ is split, we can find a partition of $V(G)$ into a clique and independent set in polynomial time.
    \end{lemma}
    \begin{proof}
        Let $d_1\ge d_2 \ldots \ge d_n$ be the degree sequence of $G$ and let $v_1,v_2,...,v_n$ be a labeling of vertices such that $v_i$ has degree $d_i$. Hammer and Simeone \cite{HS81} showed that a graph is a split iff for largest index $k$ such that $d_k\ge k-1$, we have:
        $$k(k-1)-\sum_{i=1}^k d_i+\sum_{i=k+1}^n d_i=0.$$
        Furthermore the authors showed that if $G$ is split, then $(\bigcup_{i=1}^k v_i,\bigcup_{j=k+1}^nv_j)$ is a partition of $G$ into a clique and an independent set. Note that the above-mentioned process can clearly be done in polynomial time.
    \end{proof}
       
    \begin{theorem}\label{thm:NPhardness}
        Let $G= (V,E)$ be a graph. For any fixed integer $t\ge 1$  finding a maximum $t$-equidistant number of $G$ is an NP-hard problem. 
    \end{theorem}
    \begin{proof}
        If $G$ is split, then determining the clique number of $G$ can be done in polynomial time. Let $(V_1,V_2)$ be a partition of the vertices of $G$ such that $V_1$ is a clique and $V_2$ is an independent set. Note that we can obtain such a partition (if it exists) in polynomial time using Lemma~\ref{lem:SplitCheck}. Note that as $V_2$ is an independent set of nodes any clique of $G$ can contain at most one vertex in $V_2$. Furthermore, we already know $V_1$ is a clique. Hence  if $V_1 \cup \{v\}$ is a clique for any $v \in V_2$ then this is the maximum clique and otherwise $V_1$ itself is the maximum clique of $G$.

        Now assume $G$ is not split. We will give a reduction of determining the clique number of $G$ to determining the $t$-equidistant number of a graph $H$. We split into two cases, $t$ even and $t$ odd.
        \begin{itemize}
            \item \textbf{$t$ is odd:} Construct a graph $H=(V',E')$ from $G$ by replacing every edge $e \in E$ by a path of length $t$, that is add $t-1$ additional vertices into every edge. We henceforth call the vertices in $V$ original vertices, and the vertices in $V' \setminus V$ dummy vertices. Note that the distance of original vertices $u$ and $v$ is $d$ in $G$ then their distance is $k\times t$ in $H$. 
            
            We now want to show that $G$ has a clique of size $m$ if and only if $H$ has a $t$-equidistant set of size $m$. The first implication follows directly from the definition of $H$. For the second, note that:
            \begin{enumerate}[label=(\roman*)]
                \item If $U \subset V$ is a $t$-equidistant set in $H$, then $U$ is a clique in $G$.
                \item If $U \subset V'$ is a $t$-equidistant set in $H$, then $U$ either consists entirely of original vertices, or entirely of dummy vertices. This is due to the fact that the distance of a dummy vertex from an original vertex can not be a multiple of $t$.
            \end{enumerate}
            Meaning we only need to consider the additional case when $U \subset V' \setminus V$ is a $t$-equidistant set in $H$. For $m=1$ or $2$ this is trivial, hence assume $\abs{U} = m \geq 3$.
            
            Now, consider a $t$-equidistant set in $H$, for which $U \subset V' \setminus V$ and $\abs{U} = 3$. Consider two vertices $u_1,u_2 \in U$. By construction, they must lie on two distinct paths corresponding to two distinct edges in $G$ that share an endpoint, say $(v_1,v_2), (v_1,v_3) \in E$. Now,
            \begin{equation*}
                t= d(u_1,u_2) = d_H(u_1,v_1) + d_H(u_2,v_1) \implies d_H(u_1,v_1) \neq d_H(u_2,v_1).
            \end{equation*}
            This is because $t$ is odd. Hence, for the third vertex $u_3 \in U$ it cannot be on a path corresponding to an edge incident to $v_1$, as it could never be at distance $t$ from both $u_1$ and $u_2$ in this case. Instead, it must be on the path corresponding to the edge $(v_2,v_3)$. For this to occur, $v_1v_2v_3$ must be a triangle in $G$, meaning we have a clique of size $3$ in $G$.

            Finally, assume in the above that $m \geq 4$. Then there exists a fourth vertex $u_4 \in U$, however, just as $u_3$, this vertex must lie on the path corresponding to the edge $(v_2,v_3)$. But then $d_H(u_3,u_4) <t$ which is a contradiction and hence this case cannot occur. 

            Thus $\omega(G) = \equ_t(H)$ and hence we can find $\omega(G)$ by computing $\equ_t(H)$.

            \item \textbf{$t$ is even:} Just as in the odd case, construct a graph $H=(V',E')$ from $G$ by replacing every edge $e \in E$ by a path of length $t$, that is add $t-1$ additional vertices into every edge.  However, now we additionally connect all the "central" dummy vertices, i.e. the vertices in the center of the paths of length $t$ replacing the edges. We now want to show that $G$ has a clique of size $m$ if and only if $H$ has a $t$-equidistant set of size $m+1$. For this we need a few observations:
            \begin{enumerate}[label=(\roman*)]
                \item $U \subseteq V$ is a $t$-equidistant set in $H$ if and only if $U$ is a clique in $G$.
                \item Any $t$-equidistant set $U \subseteq V'$ in $H$ contains at most $1$ dummy vertex.
                \item For any clique $U \subseteq V$ in $G$, there exists an edge $e \in E$ such that $e$ has no endpoint in $U$.
            \end{enumerate}

            These are true because:
            \begin{enumerate}[label=(\roman*)]
                \item Follows directly from the construction of $H$.
                \item If $v \in V' \setminus V$ is a dummy vertex, then it will be at distance at most $\frac{t}{2}-1$ from a central dummy vertex. But since this holds for all dummy vertices, and all central dummy vertices form a clique, we conclude any two dummy vertices are within distance $(\frac{t}{2}-1)+(\frac{t}{2}-1)+1 = t-1$ from each other, and hence no two dummy vertices could be in the same $t$-equidistant set of $H$.
                \item By assumption,  $G$ is not split and thus if $U$ is a clique in $G$, then $V \setminus U$ is not an independent set in $G$ . In other words, there exist two vertices $u,v \in V \setminus U$ such that $(u,v) \in E$.
            \end{enumerate}

            From assertions (i) and (ii) we find $\omega(G) = \equ_t(H)$ or $\omega(G) +1= \equ_t(H)$. Assertion (iii) tells us we are always in the latter case. Indeed, let $U$ be a clique in $G$ of size $\omega(G)$, then $U$ is a $t$-equidistant set in $H$. Now consider the edge $(v_1,v_2) \in E$ such that $\{v_1,v_2\} \cap U = \emptyset$, and in particular consider its corresponding dummy vertices in $H$. Denote by $w$ the dummy vertex in $H$ for which $d(v_1,w) = 1$ and $d(v_2,w) =t-1$. Now, any vertex in $U$ will be at distance exactly $\frac{t}{2}$ to some central dummy vertex in $H$, and hence at distance $\frac{t}{2}+1$ to the central dummy vertex corresponding to $e$. This in turn means they are all at distance exactly $t$ from the dummy vertex $w$. Thus $U \cup \{w\}$ is a $t$-equidistant set in $H$ of size $\omega(G)+1$ and hence $\omega(G) +1= \equ_t(H)$. This means we can exactly determine $\omega(G)$ by computing $\equ_t(H)$, as we wanted.
        \end{itemize}

        Now since $\omega$ is NP-hard, $\equ_t$ must be NP-hard as well.
    \end{proof}

    In fact, in the next result we show that it is impossible to approximate the $t$-equidistant number with a constant factor unless $P = NP$.

\begin{theorem}
     For every fixed integer $t$ and $\epsilon\ge 0$, there exists no constant approximation for computing the $t$-equidistant number of a graph $H$, where $n$ is the number of vertices of $H$, unless $P=NP$.
\end{theorem}

\begin{proof}
    For this purpose we show that if such an approximation algorithm exist then there exists a constant approximation for the maximum clique problem where $k$ is the number of vertices of the input graph.
    Now to prove use we use a similar argument to the proof of Theorem~\ref{thm:NPhardness}. 
    Let $G$ be the input graph of a maximum clique instance. We first use Lemma~\ref{lem:SplitCheck} to check whether $G$ is split or no. If $G$ is split then we can find its maximum clique in polynomial time. Hence we assume that $G$ is not split. Assume there exists a $c$-approximation for $\equ_t$ for some constant $c>1$.
    Now we use the same reduction as discussed in the proof of Theorem~\ref{thm:NPhardness}. We distinguish the following cases:
    \begin{itemize}
        \item \textbf{$t$ is odd:} Note that our reduction constructs an instance $H$ of $G$ such that the size of $H$ is polynomial in terms of size of $G$. Furthermore for every feasible solution of $G$ we can compute in polynomial time a feasible solution with the same size for $H$ and vice versa.
        Therefore given a $c$-approximation for $H$, we can obtain a $c$-approximation for $G$ as well.
        \item \textbf{$t$ is even:} Again our reduction constructs an instance $H$ of $G$ such that the size of $H$ is polynomial in terms of size of $G$. However now for every feasible solution of $H$ we can compute in polynomial time a feasible solution that is at most one unit smaller for $G$. Furthermore we know that $\omega(G)=\equ_t(H)+1$.  
        Now we can use the $c$-approximation algorithm for maximum $t$-equidistant set to compute a $t$-equidistant set $S$ of size $\equ_t(H)/c$. Then we can use this to find a clique $C$ of size at least $\equ_t(H)/c-1$ in $G$. Observe that $\equ_t(H)/c-1=(\omega(G)+1)/c-1$. Now as $G$ is not split it has at least one edge. Therefore we can compute a clique of size $\max\{2,(\omega(G)+1)/c-1\}>\omega(G)/2c$. Thus this leads to a $2c$-approximation for maximum clique problem.
        
    \end{itemize}
    However under the assumption of $P\neq NP$ there exists no approximation for the clique number within a constant factor (\cite{D2007}), hence the claim.    
\end{proof}

We now provide a proof for the NP-hardness of $\equ$. This follows from a fairly simple graph construction.

\begin{theorem}
        Let $G= (V,E)$ be a graph. Finding the equidistant number of $G$ is an NP-hard problem. 
    \end{theorem}
    \begin{proof}
        We will give a reduction of determining the clique number of $G$ to determining the equidistant number of a graph $H$.
        Let $G=(V,E)$ with $\abs{V} = n$, and let $K_n=(U,F)$ be the complete graph on $n$ vertices. Then consider $H = G \nabla K_n$, the join of $G$ and $K_n$. That is, $H$ has vertex set $V \cup U$ and edge set $E \cup F \cup \{(v,u): v \in V, u \in U\}$.
        First we show $\omega(H) \geq \equ_t(H)$ for all $t$, which implies $\equ(H) = \omega(H)$. For any $v \in V$ we have $U \cup \{v\}$ is a clique in $H$ of size $n+1$, hence $\omega(H) \geq n+1$. On the other hand, for $t >1$ we have that any $t$-equidistant set can contain at most 1 vertex in $U$. Hence $\equ_t(H) \leq n+1$, and thus $\equ(H) = \omega(H)$.
        Finally, we will show $\omega(H) = \omega(G) + n$. Clearly if we have a clique $C$ of size $m$ in $G$, then $C \cup U$ will be a clique of size $m+n$ in $H$, hence $\omega(H) \geq \omega(G) + n$. On the other hand, if $C'$ is a maximal clique of size $m'$ in $H$, then $U \subset C'$ because all vertices in $U$ are adjacent to all vertices in $V \cup U$. Thus $C = C' \setminus U$ will be a clique in $G$ of size $m'-n$, and hence  $\omega(H) \leq \omega(G) + n$.
        Thus, we have found $\equ(H) = \omega(H) = \omega(G) + n$, which implies that we can determine the clique number in polynomial time, if we are able to determine the equidistant number in polynomial time. Hence, determining the equidistant number is an NP-hard problem.
    \end{proof}

\section{The relation between the equidistant number and the $t$-independence number}\label{sec:equ_connect}
    A set $U$ of vertices of a  graph $G$ is called a \emph{$t$-independent set} if $d(u,v) > t$ for all $u,v \in U$. The \emph{$t$-independence number} of a graph $G$, denoted $\alpha_t$, is the size of the largest $t$-independent set of $G$. Clearly, the $t$-equidistant number is closely related to the $(t-1)$-independence number, but with the underlying set having stricter requirements. The following lemmas make this connection more precise.

    The following lemmas investigate the relation between the above parameters.

    \begin{lemma}\label{lem:eqwalpha}\cite{privatecommunication}
    For any graph $G$ with clique number $\omega (G)$ we have $\omega (G) = \equ _1(G) \leq \equ (G) \leq \max \{ \omega (G),\alpha (G)\} $. For graphs 
    with diameter two we have the equality $\equ (G) = \max \{ \omega (G),\alpha (G)\} $.
    \end{lemma} 	

  A straightforward extension of Lemma \ref{lem:eqwalpha}, using $\alpha_t(G)$, is as follows.
  
    \begin{lemma}\label{lem:eqwalphat}
        Let $G$ be a graph. Then, at least one of the following holds
        \begin{enumerate}
            \item $\equ(G) = \equ_i(G)$, for some $i \in [1,t]$,
            \item $\equ(G) \leq \alpha_t(G)$.
        \end{enumerate}
    \end{lemma}

    \begin{lemma}\label{lem:equ_alpha_ineq}
        Let $G$ be a graph. Then for $t,t^* \in \mathbb{N}^+$ with $t < t^*$ we have $\alpha_{t}(G)\geq \equ_{t^*}(G)$. 
    \end{lemma}
    \begin{proof}
        If $U$ is a $t^*$-equidistant set, then all its vertices are at pairwise distance at least $t$, meaning it is also a $t$-independent set. Hence $\alpha_{t}(G)\geq \equ_{t^*}(G)$.
    \end{proof}

    \begin{lemma}\label{lem:equ_alpha_diameter}
        Let $G$ be a graph. If $t \geq \diam(G)$, then $\alpha_{t-1}(G) = \equ_t(G)$.
    \end{lemma}
    \begin{proof}
        The case $t > \diam(G)$ is trivial, as in this case no two vertices are far enough apart to form a $(t-1)$-independent set or $t$-equidistant set, and hence $\alpha_{t-1}(G) = 1 = \equ_t(G)$.
        
        Next is $t=\diam(G)$. $\alpha_{t-1}(G) \geq \equ_t(G)$ follows from Lemma \ref{lem:equ_alpha_ineq}. To prove the converse, let $U$ be a $(t-1)$-independent set. Then for all $u,v \in U$ we have $d(u,v) > t-1$, additionally, since $t=\diam(G)$, we have $d(u,v) \leq \diam(G) = t$. Thus $d(u,v) = t$ for all $u,v \in U$, and thus $U$ is a $t$-equidistant set. Since this holds for arbitrary $U$, we conclude $\alpha_{t-1}(G) \leq \equ_t(G)$, and thus $\alpha_{t-1}(G)= \equ_t(G)$.
    \end{proof}

    \begin{lemma}
        Let $G$ be a graph. Then there exists a $t^* \in \mathbb{N}^+$ such that for all $t < t^*$ we have $\alpha_t(G) \geq \equ(G)$.
    \end{lemma}
    \begin{proof}
        By definition, there is a $t^* \in \mathbb{N}^+$ such that $\equ(G) = \equ_{t^*}(G)$. The desired inequality follows from applying Lemma \ref{lem:equ_alpha_ineq}.
    \end{proof}

    By way of the above lemmas, we are able to apply eigenvalue bounds for the independence number, such as the ones seen in \cite{ACF2019} and its optimization from \cite{ACFNS2022}, to the equidistant number. We will investigate this further in Section \ref{sec:spec_equ}.

  Thanks to Lemma \ref{lem:equ_alpha_ineq}, we have an easy way to apply bounds on $\alpha_{t-1}$ to $\equ_t$. However, a quick look at the two definitions suggests that $\alpha_{t-1}$ can in general be significantly larger than $\equ_t$. Thus, next we will investigate how this gap behaves, and we will also look at the extremal cases, i.e.\ the graphs of given order for which $\alpha_{t-1} -\equ_t$ is maximal.

    First, we define the main quantity we are interested in:
    \begin{definition}
        Let $k \geq 1$, $t\geq 2$, we denote the maximum distance between $\alpha_{t-1}$ and $\equ_t$ over all connected graphs of order $k$ by:
        \begin{equation*}
            AE(k,t) = \max\{\alpha_{t-1}(G) - \equ_t(G) : G \text{ has order } k\}.
        \end{equation*}

    \end{definition}

    In particular, we will focus on studying
    \begin{equation}\label{eq:AE_lim}
            \lim_{k \rightarrow \infty} \frac{AE(k,t)}{k}.
            \end{equation}
 
    The following result shows that \eqref{eq:AE_lim} is entirely determined by how fast the $t$-independence number grows relative to the amount of vertices.

    \begin{theorem}\label{thm:AE_growth}        \begin{equation*}
    \lim_{k \rightarrow \infty} \frac{AE(k,t)}{k} = \lim_{k \rightarrow \infty}\max\left\{\frac{\alpha_{t-1}(G)}{k}: G \text{ is of order } k\right\}.
\end{equation*}
    \end{theorem}
    \begin{proof}
        The inequality
        \begin{equation*}
            \lim_{k \rightarrow \infty} \frac{AE(k,t)}{k} \leq \lim_{k \rightarrow \infty}\max\left\{\frac{\alpha_{t-1}(G)}{k}: G \text{ is of order } k\right\}.
        \end{equation*}
        follows trivially from $\alpha_{t-1} - \equ_t \leq \alpha_{t-1}$.

        For the other inequality, fix some $t \in \mathbb{N}^+$. Now choose the graphs $G(n,t)$ such that $G(n,t)$ has $n$ vertices, and
        \begin{equation*}
            \alpha_{t-1}(G(n,t)) = \max\left\{\alpha_{t-1}(G): G \text{ is of order } n\right\}.
        \end{equation*}
        Next, define the graphs $H(m,n,t)$ by connecting $m$ copies of $G(n,t)$ in a line, by way of paths of length $t$. This graph has $m(n+t)-t$ vertices, for the $m$ copies of $G(n,t)$ that have $n$ vertices, and the $m-1$ paths of length $t$ connecting them. Since none of these copies are within distance $t$, we can take the union of the largest $t$-independent sets of the copies of $G(n,t)$ as a $t$-independent set of $H(m,n,t)$. Hence, $\alpha_{t-1}(H(m,n,t)) \geq m \cdot \alpha_{t-1}(G(n,t))$.

        Furthermore, note that $\equ_t(H(m,n,t))$ will be bounded by some constant $C_{n,t}$ which does not depend on $m$. This is because for any arbitrary vertex $v \in V(H(m,n,t))$, the amount of vertices within distance $t$ (and subsequently those at exactly $t$) does not grow as you add more copies of $G(n,t)$ in the way described here.

        Now, fix an $n$ and observe
        \begin{align*}
            \lim_{k \rightarrow \infty} \frac{AE(k,t)}{k} &\geq \lim_{m \rightarrow \infty} \frac{\alpha_{t-1}(H(m,n,t))-\equ_t(H(m,n,t))}{\abs{V(H(m,n,t))}}\\& \geq \lim_{m \rightarrow \infty} \frac{m\cdot \alpha_{t-1}(G(n,t))-C_{n,t}}{m(n+t)-t} \\
            &= \frac{\alpha_{t-1}(G(n,t))}{n+t}.
        \end{align*}

        Since this holds for arbitrary $n$, it must also hold in the limit.
        \begin{align*}
            \lim_{k \rightarrow \infty} \frac{AE(k,t)}{k} &\geq \lim_{n \rightarrow \infty} \frac{\alpha_{t-1}(G(n,t))}{n+t} = \lim_{n \rightarrow \infty} \frac{\alpha_{t-1}(G(n,t))}{n} \\
            &= \lim_{n \rightarrow \infty}\max\left\{\frac{\alpha_{t-1}(G)}{n}: G \text{ is of order } n\right\}
        \end{align*}
        where this last equality is by definition of $G(n,t)$.
    \end{proof} 

    In view of Theorem \ref{thm:AE_growth}, it would be of interest to know how 
    \begin{equation}\label{eq:alpha_lim}
         \lim_{k \rightarrow \infty} \max\left\{\frac{\alpha_t(G)}{k}: G \text{ is of order } k\right\}
    \end{equation}
    behaves. For $t=1$, the classical independence number, we know this limit is equal to $1$, take for instance the star graphs $S_n$. For general $t$, we determine this limit in Proposition \ref{prop:alpha_lim}
    
    \begin{proposition}\label{prop:alpha_lim}
        Let $t \leq 2k-1$. Then:
        \begin{equation*}
            \lim_{k \rightarrow \infty} \max\left\{\frac{\alpha_t(G)}{k}: G \text{ is of order } k\right\} = \frac{1}{\left\lfloor\frac{t+1}{2}\right\rfloor}.
        \end{equation*}
    \end{proposition}
    \begin{proof}
        Let $G=(V,E)$ be a graph of order $k$ and let $\alpha_t$ be its $t$-independence number. We show that 
        \begin{equation*}\label{eq:alpha_n_upper}
            \frac{\alpha_t}{k} \leq \frac{1}{\left\lfloor\frac{t+1}{2}\right\rfloor}.
        \end{equation*}    
        Clearly if $\alpha_t = 1$ the above is already satisfied, so assume $\alpha_t > 1$.    
        Let $U = \{u_1,\hdots,u_{r}\} \subset V$ be a $t$-independent set of size $r = \alpha_t$. Now for any $i \in [1,r]$ we have the following: Since $G$ is connected, and we assumed $r > 1$, we know that for every $u_i \in U$, there must be a vertex $v \in U$ such that $d(u_i,v) > t$. Hence, there must also be vertices $v_{i,1},\hdots,v_{i,{\left\lfloor\frac{t-1}{2}\right\rfloor}} \in V$ such that $d(u_i,v_{i,j}) = j$. Such $v
        _{i,j}$ exist for all $i \in [1,r]$. Suppose we have $i,i'\in[1,r]$ such that $i \neq i'$, then using the triangle inequality and the fact that $U$ is a $t$-independent set we find:
        \begin{equation*}
            t \leq d(u_i,u_{i'}) \leq d(u_i,v_{i,j}) + d(u_{i'},v_{i,j}) = j + d(u_{i'},v_{i,j}) \leq \left\lfloor\frac{t-1}{2}\right\rfloor +  d(u_{i'},v_{i,j}),
        \end{equation*}
         Hence, $d(u_{i'},v_{i,j}) > \left\lfloor\frac{t-1}{2}\right\rfloor$ for any $i' \neq i$. This means that if at least one of $i \neq i'$ or $j \neq j'$, then $v_{i,j} \neq v_{i',j'}$. But now we have identified $r(\left\lfloor\frac{t-1}{2}\right\rfloor+1)$ unique vertices in $G$, meaning $k \geq r(\left\lfloor\frac{t-1}{2}\right\rfloor+1) = r\left\lfloor\frac{t+1}{2}\right\rfloor$, and thus
        \begin{equation*}
            \frac{r}{k} \leq \frac{r}{ r\left\lfloor\frac{t+1}{2}\right\rfloor} = \frac{1}{\left\lfloor\frac{t+1}{2}\right\rfloor}.
        \end{equation*}
        To show that this upper bound is attained (asymptotically), consider the extended star graph $ES(n,m)$, which is a graph with a central vertex, to which $n-1$ paths of length $m$ are connected. Note that $ES(n,m)$ has $(n-1)m+1$ vertices. Now consider $ES(n,\left\lfloor\frac{t+1}{2}\right\rfloor)$ and observe that the outermost vertices of each path form a $t$-independent set. Hence, we find
        \begin{align*}
            \lim_{k \rightarrow \infty} \max\left\{\frac{\alpha_t(G)}{k}: G \text{ is of order } k\right\} &\geq \lim_{n \rightarrow \infty} \frac{\alpha_t(ES(n,\left\lfloor\frac{t+1}{2}\right\rfloor))}{\abs{V(ES(n,\left\lfloor\frac{t+1}{2}\right\rfloor))}} \\
            &= \lim_{n \rightarrow \infty} \frac{n-1}{(n-1)\left\lfloor\frac{t+1}{2}\right\rfloor+1}= \frac{1}{\left\lfloor\frac{t+1}{2}\right\rfloor}. 
      \qedhere  \end{align*}
    \end{proof}
    
\section{Eigenvalue bounds for the ($t$-)equidistant number}\label{sec:spec_equ}

We have seen in Section \ref{sec:complexity} that the $t$-equidistant number is hard to compute. Thus, it makes sense to derive bounds for this parameter using eigenvalues, since these can be computed in polynomial time. In particular, we will investigate how known bounds on $\alpha_t$ and $\omega$ can be used to derive bounds on $\equ_t$ and $\equ$. In addition, we will introduce a completely different bound using the distance matrix spectra (Theorem \ref{thm:equt_distance_bound}). We will also show graph classes that attain equality for our bounds and show that our
bounds compare favorably to a previous bound on the $t$-equidistant number.

\subsection{Bounds on $\equ_t$}
As a point of comparison, we will use the following bound on $\equ_t$, which is a straightforward generalization of a bound on $\equ_2$ by Foucaud et al. \cite[Theorem 2.1]{foucaud_cliques_2021}.
    \begin{proposition} \label{prop:eqt_degree_bound}
        Let $G$ be a graph with maximum degree $\Delta$. Then
        \begin{equation}\label{eq:eqt_degree_bound}
            \equ_t(G) \leq \Delta(\Delta-1)^{t-1}+1.
        \end{equation}
    \end{proposition}
    \begin{proof}
        For some vertex $v\in V$, consider the maximum amount of vertices that can be at distance exactly $t$ from $v$, this is $\Delta(\Delta-1)^{t-1}$. Hence, the size of any $t$-equidistant set is at most $\Delta(\Delta-1)^{t-1}+1$.
    \end{proof}

    As we have seen in Section \ref{sec:equ_connect}, we can relate the equidistant number to the independence number and the clique number. There are two main ways of doing this. We can use Lemma \ref{lem:equ_alpha_ineq} and we can use Lemma \ref{lem:eqwalpha}.

    First, we will look at the bounds we can obtain on $\equ_{t+1}$ using bounds on $\alpha_t$; and we will do it applying Lemma \ref{lem:equ_alpha_ineq}. 
     
    \begin{proposition}[Inertial-type bound]\label{prop:1st_inertial_equ}
        Let $G =(V,E)$ be a graph with adjacency eigenvalues $\lambda_1 \geq \lambda_2 \geq \cdots \geq \lambda_n$ and adjacency matrix $A$. Let $p \in \mathbb{R}_t[x]$ with corresponding parameters $W(p) := \max_{u \in V} \{(p(A))_{uu}\}$ and $w(p) := \min_{u \in V} \{(p(A))_{uu}\}$. Then, the $(t+1)$-equidistant number of $G$ satisfies the bound 
        \begin{equation}\label{eq:1st_inertial_equ}
            \equ_{t+1}(G) \leq \min\{ \abs{\{i : p(\lambda_i) \geq w(p) \}},\abs{\{i : p(\lambda_i) \leq W(p) \}}\}.
        \end{equation}
    \end{proposition}
    \begin{proof}
        Use Lemma \ref{lem:equ_alpha_ineq}, and then bound $\alpha_t$ with \cite[Theorem 3.1]{ACF2019}.
    \end{proof}

 The implementation of the Inertial-type bound from Proposition \ref{prop:1st_inertial_equ} using Mixed Integer Linear Programming (MILP) is analogous to the MILP for the same bound for $\alpha_t$ \cite{ACFNS2022}. More details can be found in Section \ref{appendix:MILPinertia} of the Appendix.

    \begin{proposition}[Ratio-type bound]\label{prop:ratio_equ}
        Let $t \geq 1$ and let $G$ be a regular graph with $n$ vertices and adjacency eigenvalues $\lambda_1 \geq \lambda_2 \geq \cdots \geq \lambda_n$ and adjacency matrix $A$. Let $p \in \mathbb{R}_t[x]$ with corresponding parameters $W(p) := \max_{u \in V}$ $\{(p(A))_{uu}\}$ and $ \lambda(p):= \min_{i\in[2,n]} \{p(\lambda_i)\}$, and assume $p(\lambda_1) > \lambda(p)$. Then 
        \begin{equation}\label{eq:ratio_equ}
            \equ_{t+1}(G) \leq n\frac{W(p) - \lambda(p)}{p(\lambda_1) - \lambda(p)}.
        \end{equation}
    \end{proposition}
    \begin{proof}
        Use Lemma \ref{lem:equ_alpha_ineq}, and then bound $\alpha_t$ with \cite[Theorem 3.2]{ACF2019}.
    \end{proof}

More details on the Linear Programming (LP) implementation of this Ratio-type bound can be found in Section \ref{appendix:LPratio} from the Appendix.

For the original Ratio-type bound on $\alpha_t$ \cite[Theorem 3.2]{ACF2019}, the best choice of polynomial is known for $\alpha_2$ \cite[Corollary 3.3]{ACF2019} and $\alpha_3$ \cite{kavi_optimal_2023}. These polynomials are also the best choice for bounding $\equ_3$ and $\equ_4$ respectively. This follows because in Proposition \ref{prop:ratio_equ} we directly apply the bound from $\alpha_{t}$. If a different polynomial were to give a better bound on $\equ_{t+1}$, then it must automatically also give a better bound on $\alpha_{t-1}$, a contradiction. Thus we obtain the following two straightforward corollaries.

    \begin{corollary}[Best Ratio-type bound, $t=3$]\label{cor:ratio_equt3}
     Let $G$ be a $k$-regular graph with $n$ vertices with distinct adjacency eigenvalues $k = \theta_0 > \theta_1 > \cdots > \theta_d$ with $d \geq 2$. Let $\theta_i$ be the largest eigenvalue such that $\theta_i \leq -1$. Then, using $p(x) = x^2 -(\theta_i+\theta_{i-1})x$ in Proposition \ref{prop:ratio_equ} gives the following bound:
        \begin{equation*}
            \equ_3(G) \leq n \frac{\theta_0 + \theta_i\theta_{i-1}}{(\theta_0 - \theta_i)(\theta_0 - \theta_{i-1})}. 
        \end{equation*}
        This is the best possible bound for $t=3$ that can be obtained from Proposition \ref{prop:ratio_equ}.
    \end{corollary}

    \begin{corollary}[Best Ratio-type bound, $t=4$]\label{cor:ratio_equt4}
        Let $G$ be a $k$-regular graph with $n$ vertices with distinct adjacency eigenvalues $k = \theta_0 > \theta_1 > \cdots > \theta_d$, with $d \geq 3$. Let $\theta_s$ be the largest eigenvalue such that $\theta_s \leq - \frac{\theta^2_0 + \theta_0\theta_d-\Delta_3}{\theta_0 (\theta_d + 1)}$, where $\Delta_3 = \max_{u\in V} \{(A^3)_{uu}\}$. Let $b = -(\theta_s +\theta_{s-1} +\theta_d)$ and $c = \theta_d \theta_s + \theta_d \theta_{s-1}+\theta_s\theta_{s-1}$. Then, using $p(x) = x^3 +bx^2 +cx$ in Proposition \ref{prop:ratio_equ} gives the following bound:
        \begin{equation*}
            \equ_4(G) \leq  n \frac{\Delta_3 - \theta_0(\theta_s + \theta_{s-1} + \theta_d) - \theta_s\theta_{s-1} \theta_d}{ (\theta_0 - \theta_s)(\theta_0 - \theta_{s-1})(\theta_0 - \theta_d)}.
        \end{equation*}
        This is the best possible bound for $t=4$ that can be obtained from Proposition \ref{prop:ratio_equ}.
    \end{corollary}

    In the Appendix we will illustrate the tightness of the Ratio-type bound from Proposition \ref{prop:ratio_equ}.

   Next we show that the bounds from Corollaries \ref{cor:ratio_equt3} and \ref{cor:ratio_equt4} are tight for certain Johnson graphs. A \emph{Johnson graph}, denoted $J(n,k)$, is a graph where the vertices represent the $k$-element subsets of a set of size $n$. Two vertices are adjacent if and only if their corresponding subsets differ by exactly 1 element. The graph $J(n,k)$ has $\binom{n}{k}$ vertices and eigenvalues $\{(k-j)(n-k-j)-j : j \in [0,\min\{k,n-k\}]\}$.

    \begin{proposition}\label{prop:johnson_tight}
        The bound from Corollary \ref{cor:ratio_equt3} is tight for the Johnson graphs $J(n,3)$, $n >6$, and the bound from Corollary \ref{cor:ratio_equt4} is tight for the Johnson graphs $J(n,4)$, $n>8$.
    \end{proposition}
    \begin{proof}
        First, we will determine the exact value of $\equ_k(J(n,k))$ (note that we use $\equ_k$ here instead of $\equ_t$ to highlight the fact that it matches the $k$ from $J(n,k)$). Using the $k$-element subset representation of the vertices, we find a set $U \subset V(J(n,k))$ is a $k$-equidistant set if and only if the subsets are disjoint. From this we find $\equ_k(J(n,k)) = \left\lfloor\frac{n}{k}\right\rfloor$.

        Now we will show our bounds meet this exact value. For $J(n,3)$ we have the eigenvalues $(\theta_0,\theta_1,\theta_2,\theta_3) = (3n-9,2n-9,n-7,-3)$. Since $n >6$, we find $\theta_3$ is the largest eigenvalues less than $-1$. By plugging these values into the bound from Corollary \ref{cor:ratio_equt3} we obtain
        \begin{equation*}
             \equ_3(J(n,3)) \leq \frac{n}{3},
        \end{equation*}
 
        which of course can be rounded down to $\left\lfloor\frac{n}{3}\right\rfloor$, since $\equ_3$ is an integer.

        Now we move on to the case $\equ_4(J(n,4))$. Here we have eigenvalues $(\theta_0,\theta_1,\theta_2,\theta_3,\theta_4) = (4n-16,3n-16,2n-8,n-10,-4)$. Further, we have $\Delta_3 = k(n-k)(n-2)$, which can be obtained by a simple counting argument. For Corollary \ref{cor:ratio_equt4} we find $\theta_s$ is the largest eigenvalues less than $n-6$. Since $n >8$, we find $\theta_s = \theta_3 = n-10$. By plugging these values into the bound from Corollary \ref{cor:ratio_equt4} we obtain
        \begin{equation*}
             \equ_4(J(n,4)) \leq \frac{n}{4},
        \end{equation*} 
 
        which can be rounded down to $\left\lfloor\frac{n}{4}\right\rfloor$, since $\equ_4$ is an integer.
    \end{proof}

    We can also make use of the relation $\equ_t(G) = \omega(G^{[\sharp t]})$ to extend a bound on the clique number by Haemers \cite[Theorem 3.5]{haemers_interlacing_1995}.

    \begin{proposition}\cite[Theorem 3.5]{haemers_interlacing_1995}\label{prop:omega_haemers}
        Let $G$ be a regular graph with adjacency eigenvalues $\lambda_1 \geq \cdots \geq \lambda_n$. Then the clique number of $G$ satisfies:
        \begin{equation*}
            \omega(G) \leq n\frac{1+\lambda_2}{n-\lambda_1+\lambda_2}.
        \end{equation*}
    \end{proposition}

    \begin{proposition}\label{prop:omega_equt_UB}
        Let $G$ be a graph such that its  exact distance $t$-power, $G^{[\sharp t]}$ is regular. Let $G^{[\sharp t]}$ have adjacency eigenvalues $\beta_1 \geq \cdots \geq \beta_n$. Then the $t$-equidistant number of $G$ satisfies
        \begin{equation}\label{eq:omega_equt_UB}
            \equ_t(G) \leq n\frac{1+\beta_2}{n-\beta_1+\beta_2}.
        \end{equation}
    \end{proposition}
    \begin{proof}
        Apply Proposition \ref{prop:omega_haemers} to $G^{[\sharp t]}$, then use the fact that $\equ_t(G) = \omega(G^{[\sharp t]})$.
    \end{proof}
 
 Another parameter that Haemers \cite{H1997} looked at is $\Phi(G)$. Let $A,B \subset V(G)$, then $A$ and $B$ are called \emph{disconnected} if there is no edge between $A$ and $B$. The quantity $\Phi(G)$ is defined as the maximum of $\sqrt{\abs{A}\abs{B}}$ where $A,B$ are disjoint and disconnected. We can relate $\Phi(G)$ to the $t$-equidistant number and apply \cite[Theorem 2.4]{H1997} to obtain a bound on $\equ_t(G)$. First, we state the original bound on $\Phi(G)$:

    \begin{proposition}\cite[Theorem 2.4]{H1997}\label{prop:Phi_haemers}
        Let $G$ be a graph with Laplacian eigenvalues $\mu_1 \leq \cdots \leq \mu_n$, then
        \begin{equation*}
            \Phi(G) \leq  \frac{n}{2}\left(1-\frac{\mu_2}{\mu_n}\right).
        \end{equation*}
    \end{proposition}

    \begin{proposition}\label{prop:equ_Phi_bound}
    Let $G$ be a graph, and let $\mu_1 \leq \cdots \leq \mu_n$ be the Laplacian eigenvalues of $\overline{G^{[\sharp t]}}$. Then 
    \begin{equation}\label{eq:equ_Phi_bound}
        \equ_t(G) \leq  n\left(1-\frac{\mu_2}{\mu_n}\right)+1.
    \end{equation}
    \end{proposition}
    \begin{proof}
        Let $U$ be a $t$-equidistant set of size $\equ_t(G)$. Then, in $\overline{G^{[\sharp t]}}$, any two vertices in $U$ will not be adjacent. Hence, if we split $U$ into $U_1,U_2$, such that $\abs{U_i} \geq \frac{\equ_t(G)-1}{2}$, then $U_1,U_2$ are disconnected, and hence 
        $$\frac{\equ_t(G)-1}{2} \leq \sqrt{\abs{U_1}\abs{U_2}} \leq \Phi(\overline{G^{[\sharp t]}}).$$ 
        Finally, apply Proposition \ref{prop:Phi_haemers} to $\overline{G^{[\sharp t]}}$ to obtain the desired bound.
    \end{proof}

    It is of interest to investigate how well these bounds induced from the different parameters perform in practice, especially in view of the results from Section \ref{sec:equ_connect} which tell us the gap between $\alpha_{t-1}$ and $\equ_t$ can grow very large. For this reason, it is worth considering an alternative and more direct approach to derive eigenvalue bounds. Instead of using interlacing on the adjacency matrix (as we did to derive Propositions \ref{prop:1st_inertial_equ} and \ref{prop:ratio_equ}), in the following results (and the subsequent corollaries in the next section) we apply the interlacing method to the distance matrix of a graph. This makes sense, as the adjacency eigenvalue interlacing may not fully capture the properties of an equidistant set, whereas the distance matrix might. While eigenvalue interlacing using the adjacency matrix is a widely used tool, the distance matrix is much less studied when it comes to using eigenvalue interlacing for bounding graph parameters.

    \begin{theorem}\label{thm:equt_distance_bound}
        Let $G$ be a graph, and let $D$ be its distance matrix, with eigenvalues  $\tilde{\lambda}_1 \geq \cdots \geq \tilde{\lambda}_n$. Then the $t$-equidistant number satisfies:
        \begin{equation}\label{eq:equt_distance_bound}
            \equ_t(G) \leq \abs{\left\{i : \tilde{\lambda}_i \leq -t\right\}}+1.
        \end{equation}
    \end{theorem}
    \begin{proof}
        Assume $\equ_t \geq 2$, else the result is trivial. Let $U$ be a $t$-equidistant set of size $\equ_t$. Then, we can label the vertices in such a way that $D$ has principal $\equ_t \times \equ_t$ submatrix:
        \begin{equation*}
            B = t(J-I),
        \end{equation*}
        where $J$ is the all $1$ matrix and $I$ the identity matrix, both of appropriate size. $B$ has spectrum $\{(t(\equ_t-1))^{[1]}, (-t)^{[\equ_t-1]}\}$. Cauchy interlacing  then tells us:
        \begin{equation*}
            \tilde{\lambda}_{n-\equ_t+i} \leq -t \text{ for } i \in [2,\equ_t].
        \end{equation*}
        Hence, we must have at least $\equ_t-1$ eigenvalues less than $-t$, meaning
        \begin{align*}
            \equ_t &\leq \abs{\left\{i : \tilde{\lambda}_i \leq -t\right\}}+1.
    \qedhere    \end{align*}
    \end{proof}

    Using quotient interlacing on the distance matrix, we can obtain two more bounds on the $t$-equidistant number for the socalled transmission-regular graphs. A graph is called \emph{$d$-transmission-regular} if its distance matrix has constant row sum $d$.
    \begin{theorem}\label{thm:equt_quotient_bound}
        Let $G$ be a $d$-transmission regular graph on $n$ vertices, and let $D$ be its distance matrix, with eigenvalues $\tilde{\lambda}_1 \geq \cdots \geq \tilde{\lambda}_n$. Then the $t$-equidistant number satisfies:
        \begin{subequations}
        \begin{align}
            \equ_t &\leq \frac{(\tilde \lambda_2+t)n}{tn-d+\tilde \lambda_2} &\text{ if } tn-d+\tilde \lambda_2 > 0, \label{eq:quotient_bound_1} \\
            \equ_t &\leq \frac{(\tilde \lambda_n+t)n}{tn-d+\tilde \lambda_n} &\text{ if } tn-d+\tilde \lambda_n < 0. \label{eq:quotient_bound_2}
        \end{align}
        \end{subequations}
    \end{theorem}
    \begin{proof}
        Let $U$ be a $t$-equidistant set of size $r = \equ_t$. Now partition the rows and columns of $D$ into $U$ and $V \setminus U$. Let $S$ be the normalized characteristic matrix of this partition. Then the quotient matrix $B$ will be:
        \begin{align*}
            B =S^{\top}DS = \begin{pmatrix}
                t(r-1) & \frac{1}{r} \sum_{u \in U} \sum_{v \in V \setminus U} d_{uv}  \\
               \frac{1}{n-r} \sum_{u \in U} \sum_{v \in V \setminus U} d_{uv} & \frac{1}{n-r} \sum_{u \in V \setminus U} \sum_{v \in V \setminus U} d_{uv}
            \end{pmatrix}.
        \end{align*}
        Now using that $D$ has constant row sum $d$, we find:
        \begin{align*}
            \sum_{u \in U} \sum_{v \in V \setminus U} d_{uv} &= \sum_{u \in V} \sum_{v \in V} d_{uv} - \left(\sum_{u \in V \setminus U} \sum_{v \in V} d_{uv} + \sum_{u \in U} \sum_{v \in U} d_{uv}\right) \\
            &= \sum_{u \in V} d - \left(\sum_{u \in V \setminus U} d + \sum_{u \in U} t(r-1)\right) = nd - ((n-r)d + tr(r-1)) \\
            &= rd - tr(r-1),
        \end{align*}
        and
        \begin{align*}
            \sum_{u \in V \setminus U} \sum_{v \in V \setminus U} d_{uv} &= \sum_{u \in V} \sum_{v \in V} d_{uv} - \left(\sum_{u \in U} \sum_{v \in V} d_{uv} + \sum_{u \in V \setminus U} \sum_{v \in  U}d_{uv}\right) \\
            &= nd - (rd +rd - tr(r-1)) = nd -2rd +tr(r-1).
        \end{align*}
        Hence
        \begin{align*}
            B = \begin{pmatrix}
                t(r-1) & d - t(r-1)  \\
              \frac{rd - tr(r-1)}{n-r} & \frac{nd -2rd +tr(r-1)}{n-r}
            \end{pmatrix}.
        \end{align*}
        Note that $B$ has eigenvalues $\mu_1 = d$ and
        \begin{align*}
            \mu_2 &= \tr(B) - d = t(r-1) + \frac{nd -2rd +tr(r-1)}{n-r} - d \\
            &= \frac{(t(r-1))(n-r) + nd -2rd +tr(r-1)-d(n-r)}{n-r} \\
            &=\frac{tn(r-1)-rd}{n-r} = \frac{r(tn-d) -tn}{n-r}.
        \end{align*}
        Now we find by interlacing that $\tilde \lambda_2 \geq \mu_2 \geq \tilde \lambda_n$. If $tn-d+\tilde \lambda_2 > 0$, then:
        \begin{align*}
            \tilde \lambda_2 \geq \frac{r(tn-d) -tn}{n-r} &\iff \tilde \lambda_2(n-r) \geq r(tn-d) -tn \\
            &\iff (\tilde \lambda_2+t)n \geq r(tn-d+\tilde \lambda_2) \\
            &\iff \frac{(\tilde \lambda_2+t)n}{tn-d+\tilde \lambda_2} \geq r.
        \end{align*}
        Similarly, if $tn-d+\tilde \lambda_n < 0$ then:
                \begin{align*}
            \tilde \lambda_n \leq \frac{r(tn-d) -tn}{n-r} &\iff \tilde \lambda_n(n-r) \leq r(tn-d) -tn \\
            &\iff (\tilde \lambda_n+t)n \leq r(tn-d+\tilde \lambda_n) \\
            &\iff \frac{(\tilde \lambda_n+t)n}{tn-d+\tilde \lambda_n} \geq r.
     \qedhere   \end{align*}
    \end{proof}
    Computational experiments suggest that the bound from \eqref{eq:quotient_bound_1} may be tight for the Pasechnik graphs $P(n)$.

 \subsection{Bounds on \equ}
    Using the bound on $\equ_t$ from Theorem \ref{thm:equt_distance_bound}, we can obtain a bound on $\equ$, captured in the following corollary.
    
    \begin{corollary}\label{cor:equ_distance_bound}
        Let $G$ be a graph, and let $D$ be its distance matrix, with eigenvalues $\tilde{\lambda}_1 \geq \cdots \geq \tilde{\lambda}_n$. Then the equidistant number satisfies:
        \begin{equation}\label{eq:equ_distance_bound}
            \equ(G) \leq \abs{\left\{i : \tilde{\lambda}_i \leq -1\right\}}+1.
        \end{equation}
    \end{corollary}
    \begin{proof}
        By Theorem \ref{thm:equt_distance_bound} we have
        \begin{equation*}
             \equ_t(G) \leq \abs{\left\{i : \tilde{\lambda}_i \leq -t\right\}}+1 \leq \abs{\left\{i : \tilde{\lambda}_i \leq -1\right\}}+1
        \end{equation*}
        for all $t$. Hence, $\equ(G)$ will also be bounded by the above.
    \end{proof}

    An interesting consequence of Theorem \ref{thm:equt_distance_bound}, is that if we know the bound from Theorem \ref{thm:equt_distance_bound} is tight for some $t^*$, then we can reduce the number of values for $t$ that need to be considered when determining $\equ$. This is illustrated in the following result.

    \begin{corollary}\label{cor:equ_reduce}
        Let $G$ be a graph, and let $D$ be its distance matrix, with eigenvalues $\tilde{\lambda}_1 \geq \cdots \geq \tilde{\lambda}_n$. Then, if there exists a $t^* \in [1,\diam(G)]$ such that
        \begin{equation*}
            \equ_{t^*} = \abs{\left\{i : \tilde{\lambda}_i \leq -t^*\right\}}+1,
        \end{equation*}
        then
        \begin{equation*}
            \equ(G) = \max \{\equ_t(G) : 1 \leq t \leq t^*\}.
        \end{equation*}
    \end{corollary}
    \begin{proof}
        For $t > t^*$, we have
        \begin{equation*}
            \equ_t \leq \abs{\left\{i : \tilde{\lambda}_i \leq -t\right\}}+1 \leq \abs{\left\{i : \tilde{\lambda}_i \leq -t^*\right\}}+1 = \equ_{t^*}.
        \end{equation*}
        Hence
        \begin{align*}
            \equ(G) &= \max\{\equ_t(G) : 1 \leq t \leq \diam(G)\} = \max \{\equ_t(G) : 1 \leq t \leq t^*\}.\qedhere
        \end{align*}
    \end{proof}

For obtaining the next bound on $\equ$ we will use Lemma \ref{lem:eqwalpha} combined with a bound on $\alpha$ and a bound on $\omega$. 
    
     \begin{theorem}[Ratio bound, unpublished, see e.g. \cite{H2021}] \label{thm:ratio_bound}
         Let $G$ be a regular graph with $n$ vertices and adjacency eigenvalues $\lambda_1 \geq \cdots \geq \lambda_n$. Then the independence number of $G$ satisfies:
         \begin{equation*}
             \alpha(G) \leq n \frac{-\lambda_n}{\lambda_1-\lambda_n}.
         \end{equation*}
     \end{theorem}

    For bounding $\omega$ we will use the earlier stated result by Haemers (Proposition \ref{prop:omega_haemers}).

    \begin{corollary}\label{cor:haemers_equ}
            Let $G$ be a regular graph with adjacency eigenvalues $\lambda_1 \geq \cdots \geq \lambda_n$. Then the equidistant number of $G$ satisfies:
        \begin{equation}\label{eq:haemers_equ}
            \equ(G) \leq n\cdot \max\left\{\frac{-\lambda_n}{\lambda_1-\lambda_n},\frac{1+\lambda_2}{n-\lambda_1+\lambda_2}\right\}.
        \end{equation}
    \end{corollary}
    \begin{proof}
        Use Lemma \ref{lem:eqwalpha}, and then bound $\alpha$ and $\omega$ with Theorem \ref{thm:ratio_bound} and Proposition \ref{prop:omega_haemers}, respectively.
    \end{proof}

\subsection{Bounds performance}

Finally, using a computational approach, we investigate the tightness of our bounds, how they compare with each other and with the previous bound from Proposition \ref{prop:eqt_degree_bound}. The results can be found in the Appendix \ref{appendix:boundsperformance}. 

In Tables \ref{tab:equ_2_bounds} and \ref{tab:equ_3_bounds}, the performance of the distance spectrum bounds on $\equ_t$ (Theorems \ref{thm:equt_distance_bound} and \ref{thm:equt_quotient_bound}) is compared to the performance of the induced adjacency spectrum bounds from $\alpha_t$ (Propositions \ref{prop:1st_inertial_equ} and \ref{prop:ratio_equ}) and the induced spectral bounds from $\omega$ and $\Phi$ (Propositions \ref{prop:omega_equt_UB} and \ref{prop:equ_Phi_bound}) for $t=2,3$, respectively. In Table \ref{tab:equ_bounds} the performance of the distance spectrum bound on $\equ$ (Corollary \ref{cor:equ_distance_bound}) is compared to the performance of the induced adjacency spectrum bound from $\alpha$ and $\omega$ (Corollary \ref{cor:haemers_equ}). These bounds were computed using SageMath.
    
    In these three tables, a graph name is in boldface when at least one of the new bounds is tight. Note that while Proposition \ref{prop:equ_Phi_bound} is only tight for trivial cases in Tables \ref{tab:equ_2_bounds} and \ref{tab:equ_3_bounds}, we have found several smaller graphs for which Proposition \ref{prop:equ_Phi_bound} gives a non-trivial tight bound.

\subsection*{Acknowledgements}
The authors thank James Tuite for bringing the equidistant number to their attention, Sjanne Zeijlemaker for helping with the LP from Section \ref{appendix:LPratio}, and Frits Spieksma for useful discussions on Section \ref{sec:complexity}. Aida Abiad is supported by the Dutch Research Council through the grant VI.Vidi.213.085. 



\newpage 
\section{Appendix}

\subsection{MILP for the Inertial-type bound (Proposition \ref{prop:1st_inertial_equ})}\label{appendix:MILPinertia}
Below is an MILP implementation by Abiad et al.\ \cite{ACFNS2022} for the Inertial-type bound on $\alpha_t$, which can also be used for finding the best polynomial in Proposition \ref{prop:1st_inertial_equ}. 

Let $G = (V,E)$ have adjacency matrix $A$ having distinct eigenvalues $\theta_0 > \cdots > \theta_d$ with multiplicities $\mathbf{m}=(m_0,\hdots,m_d)$. Let $u \in V$, and let $p(x) = a_tx^t + \cdots + a_0$, $\mathbf{b} = (b_0,\hdots,b_d) \in \{0,1\}^{d+1}$, $M\in \mathbb{R}$ large, and $\varepsilon \in \mathbb{R}$ small, then we have the following MILP:
    \begin{equation}\label{eq:1st_inertial_MILP}
        \boxed{\begin{aligned}
         \text{variables: } &(a_0 \hdots,a_t), (b_0,\hdots,b_d) \\
         \text{parameters: } &t, M, \varepsilon\\
         \text{input: }&\text{For a regular graph }G \text{, its adjacency matrix } A, \\
         &\text{  its spectrum } \{\theta_0^{m_0},\hdots,\theta_d^{m_d}\}, \text{ a vertex } u \in V\\
         \text{output: } &(a_0 \hdots,a_t) \text{,  the coefficients of a polynomial } p \\      
        \ \\
        \text{minimize} \ \  &\mathbf{m}^\top\mathbf{b} \\
        \text{subject to}\ \  &\sum_{i=0}^t a_i(A^i)_{vv} \geq 0, \ \ v \in V \setminus \{u\} \\
        &\sum_{i=0}^t a_i(A^i)_{uu} = 0 \\
        &\sum_{i=0}^t a_i\theta_j^i - Mb_j + \varepsilon \leq 0, \ j=0,\hdots,d \\
        & \mathbf{b} \in \{0,1\}^{d+1}
        \end{aligned}}
    \end{equation}
    Implementing such an MILP for all $u \in V$ will find the best $p$ for $G$. Note that when the graph is walk-regular, this MILP only needs to be run once.

\newpage

\subsection{LP for the Ratio-type bound (Proposition \ref{prop:ratio_equ})}\label{appendix:LPratio}
The following LP can be used to compute the optimal polynomial of Proposition \ref{prop:ratio_equ}.

    Let $G = (V,E)$ have adjacency matrix $A$ and distinct eigenvalues $\theta_0 > \cdots > \theta_d$. Then, for some $u \in V$ and $l \in [1,d]$ we have the LP:
    
    \begin{equation}\label{eq:ratio_chi_MILP}
        \boxed{\begin{aligned}
            \text{variables: } &(a_0 \hdots,a_t) \\
            \text{parameters: } &t\\
            \text{input: }&\text{For a graph }G \text{, its adjacency matrix } A, \\
            &\text{The spectrum of } A: \{\theta_0,\hdots,\theta_d\}.\\
            &\text{A vertex } u \in V. \text{ An } l \in [1,d].\\
            \text{output: }&(a_0 \hdots,a_t) \text{, the coefficients of a polynomial } p \\
            \ \\
            \text{maximize} \ \ &\sum_{i=0}^t a_i \theta_0^i -\sum_{i=0}^t a_i\theta_l^i\\
            \text{subject to} \ \ & \sum_{i=0}^t a_i((A^i)_{vv} - (A^i)_{uu}) \leq 0, \ v \in V \setminus \{u\} \\
            &\sum_{i=0}^t a_i((A^i)_{uu} - \theta_l^i) = 1\\
            &\sum_{i=0}^t a_i(\theta_0^i - \theta_j^i) > 0, \ \ j \in [1,d] \\
            &\sum_{i=0}^t a_i(\theta_j^i - \theta_l^i) \geq 0, \ \ j \in [1,d]
        \end{aligned}}
    \end{equation}
    Solving this for all $u \in V$ and $l \in [1,d]$ yields the best polynomial.

\newpage

\subsection{The performance of the eigenvalue bounds from Section \ref{sec:spec_equ}}\label{appendix:boundsperformance}

\begin{table}[!htp]
    \centering
    \tiny
    \begin{tabular}{l|cccccccc|c}
        \hline
        Graph & \eqref{eq:eqt_degree_bound} & \eqref{eq:1st_inertial_equ} & \eqref{eq:ratio_equ} & \eqref{eq:omega_equt_UB} & \eqref{eq:equ_Phi_bound} & \eqref{eq:equt_distance_bound} & \eqref{eq:quotient_bound_1} & \eqref{eq:quotient_bound_2} & $\equ_2$ \\
        \hline
        Balaban 10-cage & $7$ & $36$ & $35$ & $7$ & $10$ & $37$ & - & $13$ & $3$ \\
        Frucht graph & $7$ & $6$ & $5$ & - & $41$ & $4$ & - & - & $3$ \\
        Meredith Graph & $13$ & $45$ & $34$ & $6$ & $9$ & $40$ & - & - & $5$ \\
        Moebius-Kantor Graph & $7$ & $8$ & $8$ & $11$ & $36$ & $9$ & - & $9$ & $4$ \\
        Bidiakis cube & $7$ & $8$ & $5$ & - & $42$ & $4$ & - & - & $3$ \\
        \textbf{Gosset Graph} & $703$ & $8$ & $5$ & $7$ & $27$ & $8$ & $\textbf{4}$ & - & $4$ \\
        Balaban 11-cage & $7$ & $59$ & $54$ & $6$ & $6$ & $52$ & - & - & $3$ \\
        Moser spindle & $13$ & $3$ & - & - & $54$ & $3$ & - & - & $2$ \\
        Gray graph & $7$ & $35$ & $27$ & $7$ & $12$ & $25$ & - & - & $3$ \\
        Nauru Graph & $7$ & $14$ & $12$ & $8$ & $27$ & $16$ & - & $8$ & $3$ \\
        Blanusa First Snark Graph & $7$ & $8$ & $8$ & $5$ & $29$ & $7$ & - & - & $4$ \\
        \textbf{Grotzsch graph} & $21$ & $\textbf{5}$ & - & - & $47$ & $6$ & - & - & $5$ \\
        Pappus Graph & $7$ & $11$ & $9$ & $10$ & $36$ & $7$ & - & $7$ & $3$ \\
        Blanusa Second Snark Graph & $7$ & $8$ & $7$ & $4$ & $27$ & $6$ & - & - & $3$ \\
        \textbf{Hall-Janko graph} & $1261$ & $37$ & $\textbf{10}$ & $\textbf{10}$ & $18$ & $37$ & $\textbf{10}$ & - & $10$ \\
        \textbf{Poussin Graph} & $31$ & $6$ & - & - & $38$ & $\textbf{4}$ & - & - & $4$ \\
        Brinkmann graph & $13$ & $9$ & $8$ & $8$ & $43$ & $10$ & $21$ & $21$ & $5$ \\
        Harborth Graph & $13$ & $20$ & $20$ & - & $17$ & $15$ & - & - & $4$ \\
        Perkel Graph & $31$ & $20$ & $19$ & $7$ & $21$ & $19$ & - & $19$ & $6$ \\
        Brouwer-Haemers & $381$ & $20$ & $21$ & $21$ & $24$ & $61$ & $21$ & - & $15$ \\
        Harries Graph & $7$ & $35$ & $35$ & $7$ & $9$ & $31$ & - & $13$ & $3$ \\
        \textbf{Petersen graph} & $7$ & $\textbf{4}$ & $\textbf{4}$ & $\textbf{4}$ & $43$ & $6$ & $4$ & - & $4$ \\
        Bucky Ball & $7$ & $30$ & $27$ & $5$ & $10$ & $25$ & - & $17$ & $3$ \\
        Harries-Wong graph & $7$ & $35$ & $35$ & $7$ & $9$ & $31$ & - & $13$ & $3$ \\
        \textbf{Robertson Graph} & $13$ & $8$ & $\textbf{7}$ & $9$ & $49$ & $10$ & $12$ & $40$ & $7$ \\
        \textbf{Heawood graph} & $\textbf{7}$ & $\textbf{7}$ & $\textbf{7}$ & $14$ & $36$ & $8$ & $21$ & $10$ & $7$ \\
        \textbf{Schläfli graph} & $241$ & $7$ & $\textbf{3}$ & $\textbf{3}$ & $24$ & $7$ & $\textbf{3}$ & - & $3$ \\
        \textbf{Herschel graph} & $13$ & $7$ & - & - & $39$ & $\textbf{6}$ & - & - & $6$ \\
        \textbf{Shrikhande graph} & $31$ & $7$ & $\textbf{4}$ & $\textbf{4}$ & $36$ & $7$ & $\textbf{4}$ & - & $4$ \\
        \textbf{Higman-Sims graph} & $463$ & $\textbf{22}$ & $26$ & $26$ & $24$ & $78$ & $26$ & - & $22$ \\
        \textbf{Hoffman Graph} & $13$ & $11$ & $\textbf{8}$ & - & $36$ & $\textbf{8}$ & - & - & $8$ \\
        Sousselier Graph & $21$ & $6$ & - & - & $45$ & $7$ & - & - & $5$ \\
        \textbf{Clebsch graph} & $21$ & $\textbf{5}$ & $6$ & $6$ & $36$ & $11$ & $6$ & - & $5$ \\
        \textbf{Hoffman-Singleton graph} & $43$ & $21$ & $\textbf{15}$ & $\textbf{15}$ & $36$ & $29$ & $15$ & - & $15$ \\
        Sylvester Graph & $21$ & $17$ & $13$ & $9$ & $29$ & $17$ & - & $16$ & $5$ \\
        Coxeter Graph & $7$ & $13$ & $12$ & $4$ & $16$ & $9$ & - & $7$ & $3$ \\
        Holt graph & $13$ & $10$ & $10$ & $5$ & $27$ & $11$ & - & $14$ & $4$ \\
        Szekeres Snark Graph & $7$ & $24$ & $23$ & $5$ & $11$ & $17$ & - & - & $4$ \\
        Desargues Graph & $7$ & $10$ & $10$ & $9$ & $29$ & $6$ & - & $9$ & $4$ \\
        Horton Graph & $7$ & $50$ & $48$ & $7$ & $7$ & $40$ & - & - & $4$ \\
        \textbf{Thomsen graph} & $7$ & $5$ & $\textbf{3}$ & $6$ & $36$ & $5$ & $3$ & - & $3$ \\
        Tietze Graph & $7$ & $6$ & $5$ & - & $41$ & $6$ & - & - & $4$ \\
        Double star snark & $7$ & $15$ & $13$ & $5$ & $17$ & $10$ & - & $12$ & $3$ \\
        Krackhardt Kite Graph & $31$ & $4$ & - & - & $49$ & $4$ & - & - & $3$ \\
        Durer graph & $7$ & $7$ & $5$ & - & $39$ & $4$ & - & - & $3$ \\
        Klein 3-regular Graph & $7$ & $29$ & $25$ & $4$ & $9$ & $16$ & - & $10$ & $3$ \\
        \textbf{Truncated Tetrahedron} & $7$ & $6$ & $4$ & $\textbf{3}$ & $29$ & $4$ & $20$ & $9$ & $3$ \\
        Dyck graph & $7$ & $16$ & $16$ & $8$ & $18$ & $17$ & - & $9$ & $3$ \\
        \textbf{Klein 7-regular Graph} & $43$ & $9$ & $\textbf{6}$ & $\textbf{6}$ & $30$ & $9$ & $9$ & - & $6$ \\
        Tutte 12-Cage & $7$ & $77$ & $63$ & $7$ & $5$ & $71$ & - & $10$ & $3$ \\
        Ellingham-Horton 54-graph & $7$ & $29$ & $27$ & $7$ & $12$ & $24$ & - & - & $4$ \\
        Tutte-Coxeter graph & $7$ & $20$ & $15$ & $8$ & $22$ & $19$ & - & $7$ & $3$ \\
        Ellingham-Horton 78-graph & $7$ & $40$ & $39$ & $7$ & $9$ & $30$ & - & - & $4$ \\
        Ljubljana graph & $7$ & $63$ & $56$ & $7$ & $6$ & $43$ & - & - & $3$ \\
        Errera graph & $31$ & $7$ & - & - & $35$ & $5$ & - & - & $4$ \\
        \textbf{Wagner Graph} & $7$ & $\textbf{3}$ & $\textbf{3}$ & $\textbf{3}$ & $45$ & $5$ & $\textbf{3}$ & - & $3$ \\
        F26A Graph & $7$ & $13$ & $13$ & $8$ & $23$ & $7$ & - & $8$ & $3$ \\
        \textbf{M22 Graph} & $241$ & $\textbf{21}$ & $\textbf{21}$ & $\textbf{21}$ & $26$ & $56$ & $21$ & - & $21$ \\
        Flower Snark & $7$ & $10$ & $9$ & $5$ & $27$ & $8$ & - & - & $4$ \\
        Markstroem Graph & $7$ & $10$ & $10$ & - & $20$ & $7$ & - & - & $3$ \\
        Wells graph & $21$ & $13$ & $12$ & $10$ & $36$ & $9$ & $167$ & $24$ & $5$ \\
        \textbf{Folkman Graph} & $13$ & $15$ & $\textbf{10}$ & - & $36$ & $\textbf{10}$ & - & - & $10$ \\
        Wiener-Araya Graph & $13$ & $19$ & - & - & $17$ & $14$ & - & - & $4$ \\
        Foster Graph & $7$ & $50$ & $45$ & $7$ & $7$ & $42$ & - & $13$ & $3$ \\
        McGee graph & $7$ & $12$ & $11$ & $5$ & $22$ & $10$ & - & $9$ & $3$ \\
        \textbf{Hexahedron} & $7$ & $\textbf{4}$ & $\textbf{4}$ & $8$ & $36$ & $\textbf{4}$ & $\textbf{4}$ & - & $4$ \\
        \textbf{Dodecahedron} & $7$ & $11$ & $8$ & $\textbf{3}$ & $22$ & $8$ & - & $9$ & $3$ \\
        \textbf{Octahedron} & $13$ & $4$ & $\textbf{2}$ & $\textbf{2}$ & $24$ & $4$ & $\textbf{2}$ & - & $2$ \\
        \textbf{Icosahedron} & $21$ & $4$ & $\textbf{3}$ & $4$ & $34$ & $4$ & $\textbf{4}$ & - & $3$ \\
        \hline
    \end{tabular}
    \caption{Comparison of $\equ_2$ bounds for Sage named graphs. A ``-" indicates the graph does not satisfy the conditions for the bound to be applicable. Graph names are in bold when one of the new bounds is tight.}
    \label{tab:equ_2_bounds}
\end{table}

\begin{table}[!htp]
    \centering
    \tiny
    \begin{tabular}{l|cccccccc|c}
        \hline
        Graph & \eqref{eq:eqt_degree_bound} & \eqref{eq:1st_inertial_equ} & \eqref{eq:ratio_equ} &  \eqref{eq:omega_equt_UB} & \eqref{eq:equ_Phi_bound} & \eqref{eq:equt_distance_bound} & \eqref{eq:quotient_bound_1} & \eqref{eq:quotient_bound_2} & $\equ_3$ \\
        \hline
        Balaban 10-cage & $13$ & $19$ & $17$ & $5$ & $19$ & $28$ & - & $22$ & $2$ \\
        \textbf{Frucht graph} & $13$ & $\textbf{3}$ & $\textbf{3}$ & - & $39$ & $4$ & - & - & $3$ \\
        Meredith Graph & $37$ & $10$ & $14$ & - & $16$ & $15$ & - & - & $2$ \\
        Moebius-Kantor Graph & $13$ & $6$ & $4$ & $3$ & $38$ & $5$ & $5$ & - & $2$ \\
        Bidiakis cube & $13$ & $4$ & $3$ & - & $36$ & $4$ & - & - & $2$ \\
        \textbf{Gosset Graph} & $18253$ & $8$ & $\textbf{2}$ & $\textbf{2}$ & $3$ & $8$ & $\textbf{2}$ & - & $2$ \\
        Balaban 11-cage & $13$ & $39$ & $27$ & $5$ & $9$ & $52$ & - & - & $2$ \\
        \textbf{Moser spindle} & $37$ & $2$ & - & $\textbf{1}$ & $\textbf{1}$ & $2$ & - & - & $1$ \\
        Gray graph & $13$ & $19$ & $11$ & $5$ & $24$ & $7$ & - & - & $2$ \\
        Nauru Graph & $13$ & $8$ & $6$ & $4$ & $46$ & $7$ & $10$ & - & $2$ \\
        Blanusa First Snark Graph & $13$ & $4$ & $4$ & - & $31$ & $5$ & - & - & $3$ \\
        \textbf{Grotzsch graph} & $81$ & $\textbf{1}$ & - & $\textbf{1}$ & $\textbf{1}$ & $6$ & - & - & $1$ \\
        Pappus Graph & $13$ & $7$ & $3$ & $3$ & $40$ & $7$ & $6$ & - & $2$ \\
        \textbf{Blanusa Second Snark Graph} & $13$ & $\textbf{4}$ & $\textbf{4}$ & - & $29$ & $5$ & - & - & $4$ \\
        \textbf{Hall-Janko graph} & $44101$ & $\textbf{1}$ & $\textbf{1}$ & $\textbf{1}$ & $\textbf{1}$ & $37$ & $3$ & - & $1$ \\
        Poussin Graph & $151$ & $4$ & - & - & $28$ & $4$ & - & - & $2$ \\
        \textbf{Brinkmann graph} & $37$ & $6$ & $\textbf{3}$ & $\textbf{3}$ & $21$ & $8$ & $4$ & - & $3$ \\
        Harborth Graph & $37$ & $13$ & $10$ & - & $21$ & $12$ & - & - & $3$ \\
        Perkel Graph & $151$ & $18$ & $\textbf{5}$ & $7$ & $16$ & $19$ & $6$ & - & $5$ \\
        \textbf{Brouwer-Haemers} & $7221$ & $\textbf{1}$ & $\textbf{1}$ & $\textbf{1}$ & $\textbf{1}$ & $61$ & $6$ & - & $1$ \\
        Harries Graph & $13$ & $18$ & $17$ & $5$ & $19$ & $27$ & - & $22$ & $2$ \\
        \textbf{Petersen graph} & $13$ & $\textbf{1}$ & $\textbf{1}$ & $\textbf{1}$ & $\textbf{1}$ & $6$ & $2$ & - & $1$ \\
        Bucky Ball & $13$ & $16$ & $14$ & $5$ & $11$ & $15$ & - & $23$ & $3$ \\
        Harries-Wong graph & $13$ & $18$ & $17$ & $5$ & $19$ & $27$ & - & $22$ & $2$ \\
        \textbf{Robertson Graph} & $37$ & $5$ & $\textbf{3}$ & $\textbf{3}$ & $15$ & $8$ & $\textbf{3}$ & - & $3$ \\
        \textbf{Heawood graph} & $13$ & $\textbf{2}$ & $\textbf{2}$ & $\textbf{2}$ & $34$ & $8$ & $3$ & - & $2$ \\
        \textbf{Schläfli graph} & $3601$ & $\textbf{1}$ & $\textbf{1}$ & $\textbf{1}$ & $\textbf{1}$ & $7$ & $\textbf{1}$ & - & $1$ \\
        Herschel graph & $37$ & $3$ & - & - & $28$ & $4$ & - & - & $2$ \\
        \textbf{Shrikhande graph} & $151$ & $\textbf{1}$ & $\textbf{1}$ & $\textbf{1}$ & $\textbf{1}$ & $7$ & $2$ & - & $1$ \\
        \textbf{Higman-Sims graph} & $9703$ & $\textbf{1}$ & $\textbf{1}$ & $\textbf{1}$ & $\textbf{1}$ & $78$ & $6$ & - & $1$ \\
        \textbf{Hoffman Graph} & $37$ & $5$ & $\textbf{2}$ & $3$ & $31$ & $5$ & - & - & $2$ \\
        Sousselier Graph & $81$ & $5$ & - & - & $36$ & $6$ & - & - & $3$ \\
        \textbf{Clebsch graph} & $81$ & $\textbf{1}$ & $\textbf{1}$ & $\textbf{1}$ & $\textbf{1}$ & $11$ & $2$ & - & $1$ \\
        \textbf{Hoffman-Singleton graph} & $253$ & $\textbf{1}$ & $\textbf{1}$ & $\textbf{1}$ & $\textbf{1}$ & $29$ & $3$ & - & $1$ \\
        \textbf{Sylvester Graph} & $81$ & $10$ & $\textbf{6}$ & $\textbf{6}$ & $14$ & $17$ & $\textbf{6}$ & - & $6$ \\
        \textbf{Coxeter Graph} & $13$ & $\textbf{7}$ & $\textbf{7}$ & $\textbf{7}$ & $22$ & $9$ & $13$ & $196$ & $7$ \\
        Holt graph & $37$ & $7$ & $4$ & $4$ & $25$ & $7$ & $4$ & - & $3$ \\
        Szekeres Snark Graph & $13$ & $13$ & $12$ & - & $15$ & $13$ & - & - & $2$ \\
        Desargues Graph & $13$ & $6$ & $5$ & $3$ & $36$ & $5$ & $6$ & $90$ & $2$ \\
        Horton Graph & $13$ & $30$ & $24$ & - & $13$ & $32$ & - & - & $2$ \\
        \textbf{Thomsen graph} & $13$ & $\textbf{1}$ & $\textbf{1}$ & $\textbf{1}$ & $\textbf{1}$ & $\textbf{1}$ & $2$ & - & $1$ \\
        \textbf{Tietze Graph} & $13$ & $4$ & $\textbf{3}$ & - & $36$ & $6$ & - & - & $3$ \\
        Double star snark & $13$ & $9$ & $7$ & $6$ & $21$ & $10$ & $23$ & $33$ & $4$ \\
        Krackhardt Kite Graph & $151$ & $4$ & - & - & $43$ & $3$ & - & - & $2$ \\
        \textbf{Durer graph} & $13$ & $3$ & $\textbf{2}$ & - & $30$ & $4$ & - & - & $2$ \\
        Klein 3-regular Graph & $13$ & $19$ & $13$ & $6$ & $13$ & $16$ & - & $19$ & $3$ \\
        \textbf{Truncated Tetrahedron} & $13$ & $4$ & $\textbf{3}$ & $\textbf{3}$ & $31$ & $4$ & $\textbf{3}$ & - & $3$ \\
        Dyck graph & $13$ & $8$ & $8$ & $4$ & $33$ & $7$ & - & $23$ & $2$ \\
        \textbf{Klein 7-regular Graph} & $253$ & $9$ & $\textbf{3}$ & $\textbf{3}$ & $9$ & $9$ & $\textbf{3}$ & - & $3$ \\
        Tutte 12-Cage & $13$ & $44$ & $28$ & $5$ & $10$ & $71$ & - & $16$ & $2$ \\
        Ellingham-Horton 54-graph & $13$ & $20$ & $13$ & - & $20$ & $16$ & - & - & $2$ \\
        Tutte-Coxeter graph & $13$ & $10$ & $6$ & $4$ & $50$ & $10$ & $20$ & $70$ & $2$ \\
        Ellingham-Horton 78-graph & $13$ & $27$ & $19$ & - & $16$ & $26$ & - & - & $2$ \\
        Ljubljana graph & $13$ & $35$ & $27$ & $5$ & $11$ & $43$ & - & - & $2$ \\
        Errera graph & $151$ & $4$ & - & - & $30$ & $4$ & - & - & $2$ \\
        \textbf{Wagner Graph} & $13$ & $\textbf{1}$ & $2$ & $\textbf{1}$ & $\textbf{1}$ & $3$ & $2$ & - & $1$ \\
        F26A Graph & $13$ & $7$ & $6$ & $3$ & $41$ & $7$ & $14$ & $46$ & $2$ \\
        \textbf{M22 Graph} & $3601$ & $\textbf{1}$ & $\textbf{1}$ & $\textbf{1}$ & $\textbf{1}$ & $56$ & $5$ & - & $1$ \\
        Flower Snark & $13$ & $7$ & $5$ & - & $31$ & $6$ & - & - & $3$ \\
        Markstroem Graph & $13$ & $7$ & $6$ & - & $25$ & $6$ & - & - & $3$ \\
        Wells graph & $81$ & $9$ & $3$ & $3$ & $13$ & $9$ & $3$ & - & $2$ \\
        Folkman Graph & $37$ & $5$ & $3$ & $4$ & $36$ & $5$ & - & - & $2$ \\
        Wiener-Araya Graph & $37$ & $12$ & - & - & $21$ & $10$ & - & - & $3$ \\
        Foster Graph & $13$ & $23$ & $22$ & $5$ & $14$ & $14$ & - & $19$ & $2$ \\
        McGee graph & $13$ & $7$ & $6$ & $7$ & $36$ & $9$ & $8$ & - & $5$ \\
        \textbf{Hexahedron} & $13$ & $\textbf{2}$ & $\textbf{2}$ & $\textbf{2}$ & $18$ & $4$ & $\textbf{2}$ & - & $2$ \\
        \textbf{Dodecahedron} & $13$ & $\textbf{4}$ & $\textbf{4}$ & $\textbf{4}$ & $21$ & $\textbf{4}$ & $6$ & $57$ & $4$ \\
        \textbf{Octahedron} & $37$ & $\textbf{1}$ & $\textbf{1}$ & $\textbf{1}$ & $\textbf{1}$ & $\textbf{1}$ & $\textbf{1}$ & - & $1$ \\
        \textbf{Icosahedron} & $81$ & $4$ & $\textbf{2}$ & $\textbf{2}$ & $12$ & $4$ & $\textbf{2}$ & - & $2$ \\
        \hline
    \end{tabular}
    \caption{Comparison of $\equ_3$ bounds for Sage named graphs. A ``-" indicates the graph does not satisfy the conditions for the bound to be applicable. Graph names are in bold when one of the new bounds is tight.}
    \label{tab:equ_3_bounds}
\end{table}

    \begin{table}[!htp]
    \tiny
        \centering
        \begin{tabular}{l|cc|c}
        \hline
        Graph & \eqref{eq:equ_distance_bound} & \eqref{eq:haemers_equ} & $\equ$ \\
        \hline
        Balaban 10-cage & $37$ & $35$ & $9$ \\
        Frucht graph & $6$ & $5$ & $3$ \\
        Meredith Graph & $46$ & $34$ & $5$ \\
        Moebius-Kantor Graph & $9$ & $8$ & $4$ \\
        Bidiakis cube & $7$ & $5$ & $3$ \\
        Gosset Graph& $8$ & $14$ & $7$ \\
        Balaban 11-cage & $60$ & $54$ & $7$ \\
        Moser spindle & $5$ & - & $3$ \\
        Gray graph & $31$ & $27$ & $6$ \\
        Nauru Graph & $16$ & $12$ & $4$ \\
        Blanusa First Snark Graph & $8$ & $8$ & $4$ \\
        Grotzsch graph& $8$ & - & $5$ \\
        Pappus Graph& $8$ & $9$ & $3$ \\
        Blanusa Second Snark Graph & $10$ & $7$ & $4$ \\
        \textbf{Hall-Janko graph} & $37$ & $\mathbf{10}$ & $10$ \\
        Poussin Graph& $7$ & - & $4$ \\
        Brinkmann graph & $12$ & $8$ & $5$ \\
        Harborth Graph & $24$ & $20$ & $4$ \\
        Perkel Graph & $19$ & $19$ & $6$ \\
        Brouwer-Haemers & $61$ & $21$ & $15$ \\
        Harries Graph& $31$ & $35$ & $10$ \\
        \textbf{Petersen graph} & $6$ & $\mathbf{4}$ & $4$ \\
        Bucky Ball & $28$ & $27$ & $3$ \\
        Harries-Wong Graph& $31$ & $35$ & $9$ \\
        \textbf{Robertson Graph} & $11$ & $\mathbf{7}$ & $7$ \\
        \textbf{Heawood Graph} & $8$ & $\mathbf{7}$ & $7$ \\
        Schläfli Graph& $7$ & $9$ & $6$ \\
        \textbf{Herschel graph}& $\mathbf{6}$ & - & $6$ \\
        \textbf{Shrikhande Graph} & $7$ & $\mathbf{4}$ & $4$ \\
        Higman-Sims Graph & $78$ & $26$ & $22$ \\
        \textbf{Sims-Gewirtz Graph} & $36$ & $\mathbf{16}$ & $16$ \\
        Chvatal Graph & $9$ & $5$ & $4$ \\
        \textbf{Hoffman Graph} & $9$ & $\mathbf{8}$ & $8$ \\
        Sousselier Graph& $7$ & - & $5$ \\
        Clebsch Graph & $11$ & $6$ & $5$ \\
        \textbf{Hoffman-Singleton Graph} & $29$ & $\mathbf{15}$ & $15$ \\
        Sylvester Graph & $26$ & $13$ & $6$ \\
        Coxeter Graph & $16$ & $12$ & $7$ \\
        Holt Graph & $11$ & $10$ & $4$ \\
        Szekeres Snark Graph& $21$ & $23$ & $5$ \\
        Desargues Graph& $6$ & $10$ & $4$ \\
        Horton Graph& $46$ & $48$ & $4$ \\
        \textbf{Thomsen graph} & $5$ & $\mathbf{3}$ & $3$ \\
        \textbf{Dejter Graph}& $\mathbf{8}$ & $56$ & $8$ \\
        Kittell Graph& $10$ & $23$ & $4$ \\
        Tietze Graph & $8$ & $5$ & $4$ \\
        Double star snark & $14$ & $13$ & $4$ \\
        Krackhardt Kite Graph& $6$ & - & $4$ \\
        Durer Graph & $5$ & $5$ & $3$ \\
        Klein 3-regular Graph & $29$ & $25$ & $7$ \\
        Truncated Tetrahedron & $9$ & $4$ & $3$ \\
        Dyck Graph & $17$ & $16$ & $4$ \\
        \textbf{Klein 7-regular Graph} & $9$ & $\mathbf{6}$ & $6$ \\
        Tutte 12-Cage & $71$ & $63$ & $9$ \\
        Ellingham-Horton 54-Graph& $26$ & $27$ & $4$ \\
        Tutte-Coxeter Graph & $19$ & $15$ & $5$ \\
        Ellingham-Horton 78-Graph& $37$ & $39$ & $5$ \\
        Ljubljana Graph& $50$ & $56$ & $8$ \\
        Errera Graph& $9$ & - & $4$ \\
        \textbf{Wagner Graph} & $6$ & $\mathbf{3}$ & $3$ \\
        F26A Graph & $13$ & $13$ & $3$ \\
        M22 Graph & $56$ & $21$ & $21$ \\
        Flower Snark & $12$ & $9$ & $4$ \\
        Markstroem Graph & $12$ & $10$ & $3$ \\
        Wells Graph & $9$ & $12$ & $5$ \\
        \textbf{Folkman Graph} & $11$ & $\mathbf{10}$ & $10$ \\
        Wiener-Araya Graph & $19$ & - & $4$ \\
        Foster Graph & $42$ & $45$ & $5$ \\
        McGee Graph & $12$ & $11$ & $5$ \\
        \textbf{Franklin Graph} & $7$ & $\mathbf{6}$ & $6$ \\
        \textbf{Hexahedron} & $\mathbf{4}$ & $\mathbf{4}$ & $4$ \\
        Dodecahedron & $8$ & $8$ & $4$ \\
        \textbf{Octahedron} & $4$ & $\mathbf{3}$ & $3$ \\
        Icosahedron & $4$ & $4$ & $3$ \\
        \hline
        \end{tabular}
        \caption{Comparison of $\equ$ bounds for Sage named graphs. A ``-" indicates the graph does not satisfy the conditions for the bound to be applicable. Graph names are in bold when one of the new bounds is tight.}
        \label{tab:equ_bounds}
    \end{table}

\end{document}